\documentclass[12pt,reqno]{amsart}
\usepackage{amsfonts}
\usepackage{hyperref}
\usepackage[all]{xy}
\theoremstyle{plain}

\newtheorem{algorithm}{Algorithm}[section]
\newtheorem{thm}{Thm}

\newtheorem{lemma}[algorithm]{Lemma}

\newtheorem{theoremlet}[thm]{Theorem}

\newtheorem{remark-cath}{Remark}

\newtheorem{proposition}[algorithm]{Proposition}
\newtheorem{remark}[algorithm]{Remark}

\title{Regularization via Cheeger Deformations}
\author{Catherine Searle}
\address{Mathematics, Statistics, and Physics Department\\
1845 Fairmount Street\\
Wichita State University\\
Wichita, Kansas 67260}
\email{searle@math.wichita.edu}
\urladdr{https://sites.google.com/site/catherinesearle1/home}
\author{Pedro Sol\'{o}rzano}\thanks{The second author is supported by IMPA and CAPES}
\address{Departamento de Matem\'{a}tica\\
Universidade Federal de Santa Catarina\\
Campus Universit\'{a}rio Trindade\\
CEP 88.040-900 Florian\'{o}polis-SC, Brasil}
\email{solorzano@member.ams.edu}
\urladdr{http://pedrosolorzano.wix.com/home}
\author{Frederick Wilhelm}
\address{Department of Mathematics\\
University of California\\
Riverside, CA 92521}
\email{fred@math.ucr.edu}
\urladdr{http://mathdept.ucr.edu/faculty/wilhelm.html}
\date{February 18, 2015}
\subjclass{53C20}
\keywords{Cheeger Deformation, Riemannian submersion, normal homogenous,
totally geodesic fibers}

\begin{document}
\begin{abstract}
We show that Cheeger deformations regularize $G$--invariant metrics in a
very strong sense.
\end{abstract}

\maketitle

In the presence of a group of isometries $G,$ Cheeger developed a method for
perturbing the metric on a non-negatively curved manifold $M$ \cite{Cheeg}.
We will show, in the curvature free setting, that this method regularizes
the metric in a very strong sense. Before stating our result we recall the
definition of a Cheeger deformation.

Let $G$ be a compact group of isometries of $\left( M,g_{M}\right) .$ Let $%
g_{\mathrm{bi}}$ be a bi-invariant metric on $G,$ and consider the one
parameter family $l^{2}g_{\mathrm{bi}}+g_{M}$ of metrics on $G\times M.$ $G$
acts on $\left( G\times M,l^{2}g_{\mathrm{bi}}+g_{M}\right) $ via 
\begin{equation*}
g(p,m)=(pg^{-1},\text{ }gm),
\end{equation*}%
which we will call the Cheeger Action.

Modding out by the Cheeger Action we obtain a one parameter family $g_{l}$
of metrics on $M\cong \left( G\times M\right) /G.$ As $l\rightarrow \infty ,$
$\left( M,g_{l}\right) $ converges to $g_{M}$ \cite{PetWilh1}.

The quotient map for the Cheeger Action is 
\begin{equation*}
q:(p,m)\mapsto pm.
\end{equation*}%
For any point $x$ in the union of the principal orbits, $M^{\text{reg}},$ we
define

\begin{equation*}
\tilde{g}_{l}\equiv \frac{1}{l^{2}}g_{l}|_{T_{x}G\left( x\right)
}+g_{l}|_{T_{x}G\left( x\right) ^{\perp }},
\end{equation*}%
where $T_{x}G\left( x\right) $ is the tangent space to the orbit through $x,$
and $TG\left( x\right) ^{\perp }$ is its orthogonal complement.

\begin{theoremlet}
\label{vert rescale}Let $\left( M,g_{M}\right) $ be a complete, Riemannian $%
G $--manifold with $G$ a compact Lie group. For any non-negative integer $p%
\mathbb{\ }$and any $G$--invariant, pre-compact open subset $\mathcal{U}%
\subset M^{\text{reg}},$ as $l\rightarrow 0$ the one parameter family $%
\left\{ \tilde{g}_{l}|_{\mathcal{U}}\right\} _{l>0}$ converges in the $C^{p}$%
--topology to a $G$--invariant metric $\tilde{g}$ so that the Riemannian
submersion $\left( \mathcal{U},\tilde{g}\right) \longrightarrow \mathcal{U}%
/G $ has totally geodesic, normal homogeneous fibers.
\end{theoremlet}

The normal homogeneous metrics on the fibers have the following description.

\begin{theoremlet}
\label{vert re-scl II thm}For any $x\in \mathcal{U}$ with isotropy $G_{x},$
let $\Phi _{x}:G/G_{x}\longrightarrow G\left( x\right) $ be the $G$%
--equivariant diffeomorphism, $\Phi _{x}\left( gG_{x}\right) =gx.$ Let $g_{%
\mathrm{nh,x}}$ be the normal homogeneous metric on $G/G_{x}$ induced by the
submersion $\left( G,g_{\mathrm{bi}}\right) \longrightarrow G/G_{x}.$ Then $%
\Phi _{x}:\left( G/G_{x},\,g_{\mathrm{nh,x}}\right) \longrightarrow \left( 
\mathcal{U},\,\tilde{g}\right) $ is a Riemannian embedding whose image is
totally geodesic.
\end{theoremlet}

\begin{remark-cath}
While the embedding $\Phi _{x}:\left( G/G_{x},\,g_{\mathrm{nh,x}}\right)
\longrightarrow \left( \mathcal{U},\,\tilde{g}\right) $ preserves the
Riemannian metric and has totally geodesic image, it need not be an isometry
in the metric space sense, that is, the intrinsic and extrinsic metrics on
the orbits need not coincide. Consider a \textquotedblleft
Berger\textquotedblright\ sphere obtained by expanding the constant
curvature $1$ metric in the Hopf directions, and leaving the metric on the
horizontal distribution unchanged. It follows that the Hopf semi-circles
between pairs of antipodal points have length $>\pi .$ Since every geodesic
which is horizontal for the Hopf fibration connects antipodal points, the
extrinsic distance between any pair of antipodal points is $\leq \pi ,$ and
so the intrinsic and extrinsic metrics on the Hopf fibers are different.
\end{remark-cath}

\begin{remark-cath}
The class, $\mathcal{P},$ of principal $G$--manifolds with totally geodesic,
normal homogeneous orbits is invariant under Cheeger deformation. Theorems A
and B say that all principal $G$--manifolds are attracted to $\mathcal{P}$
by Cheeger deformations.
\end{remark-cath}

As Cheeger deformations and $G$--manifolds have been extensively studied,
others may be aware of Theorems A and B. The closest result that we found in
the literature, due to Schwachh\"{o}fer and Tapp, is Proposition 1.1 in \cite%
{SchTapp}, which deals with the case of Cheeger deforming a homogeneous space%
$,$ $M=G/H$, via $G.$

We believe there are many potential applications of Theorems \ref{vert
rescale} and \ref{vert re-scl II thm}. For example, some of the curvature
estimates in \cite{SearleWilh} can be obtained by combining Theorems \ref%
{vert rescale} and \ref{vert re-scl II thm} with the Gray--O'Neill
fundamental equations of a submersion \cite{Gray}, \cite{O'Neill}.

The paper is organized as follows. In Section 1, we establish our notations
and conventions, and in Section 2, we prove Theorems \ref{vert rescale} and %
\ref{vert re-scl II thm}.

\noindent \textbf{Acknowledgments: }We are grateful to Peter Petersen,
Wilderich Tuschmann, and Burkhard Wilking for stimulating conversations on
this topic.

\section{Notations and Conventions}

\numberwithin{equation}{algorithm}

Throughout we assume that the compact Lie group, $G,$ with bi-invariant
metric $g_{\mathrm{bi}}$, acts isometrically on the complete Riemannian
manifold $\left( M,g_{M}\right) .$ The orbit through $x\in M$ is called $%
G\left( x\right) $ and the isotropy subgroup at $x$ is $G_{x}.$ We denote
the Lie algebra of $G$ by $\mathfrak{g},$ and the Lie algebra of $G_{x}$ by $%
\mathfrak{g}_{x}.$ We call $\mathfrak{m}_{x}$ the orthogonal complement,
with respect to $g_{\mathrm{bi}},$ of $\mathfrak{g}_{x}$ in $\mathfrak{g}.$
For the distribution on $M^{\text{reg}}$ given by the tangent spaces to the
orbits of $G,$ we write $T\left( \mathrm{orbits}\right) .$

For an abstract $G$--manifold, $N,$ let 
\begin{equation}
K_{N}:\mathfrak{g}\times N\longrightarrow TN  \label{K_N dfn}
\end{equation}%
be the bundle map that takes $\left( k,x\right) \in \mathfrak{g}\times N$ to
the value at $x$ of the Killing field generated by $k,$ and let $%
K_{N,x}=\left. K_{N}\right\vert _{\mathfrak{g}\times \left\{ x\right\} }.$
Note that the map $K_{N}$ depends not just on $N$, but on the particular $G$%
--action on $N$. We adopt the convention that when $N=G,$ the $G$--action is
by right multiplication. The corresponding bundle map $K_{G}:\mathfrak{g}%
\times G\longrightarrow TG$ is then the trivialization of $TG$ given by the
left invariant fields.

Since $G$ is a $G$-manifold via various $G$-actions, the map $K_{G}:%
\mathfrak{g}\times G\longrightarrow TG$ has more than one possible meaning.
We adopt the convention that $K_{G}:\mathfrak{g}\times G\longrightarrow TG$
is the bundle map that corresponds to the action of $G$ on $G$ by right
multiplication.

For $x\in M^{\text{reg}},$ define $\tilde{\Phi}_{x}:G\longrightarrow G\left(
x\right) $ by $\tilde{\Phi}_{x}\left( g\right) =gx.$ Let $\pi
:G\longrightarrow G/G_{x}$ be the quotient map$,$ and let $\Phi
_{x}:G/G_{x}\longrightarrow G\left( x\right) $ be the $G$--equivariant
diffeomorphism given by $\Phi _{x}\left( gG_{x}\right) =gx.$ Since $\Phi
_{x}\circ \pi =\tilde{\Phi}_{x},$ $D\pi _{e}=K_{G/G_{x},\text{ }eG_{x}}$ and 
$\left( D\tilde{\Phi}_{x}\right) _{e}=K_{M,x},$ the chain rule gives%
\begin{equation*}
\left( D\Phi _{x}\right) _{eG_{x}}\circ K_{G/G_{x},\text{ }eG_{x}}=K_{M,x}.
\end{equation*}%
Since $\left. K_{G/G_{x},\text{ }eG_{x}}\right\vert _{\mathfrak{m}_{x}}$ is
invertible, 
\begin{equation}
\left( D\Phi _{x}\right) _{eG_{x}}=K_{M,x}\circ \left. K_{G/G_{x},\text{ }%
eG_{x}}\right\vert _{\mathfrak{m}_{x}}^{-1}.  \label{DPhi eqn}
\end{equation}

Note that the differential of the quotient map 
\begin{equation*}
q:(p,m)\mapsto pm
\end{equation*}%
for the Cheeger Action, $g(p,m)=(pg^{-1},$ $gm),$ is%
\begin{equation}
Dq_{(p,m)}\left( k,v\right) =K_{M,x}\left( k\right) +v.  \label{difff of q}
\end{equation}

Recall from Chapter 2 of Hirsch \cite{Hir} that two smooth maps $\Phi ,\Psi
:M\longrightarrow N$ are close in the weak $C^{p}$--topology if all of their
values and partials up to order $p$ are close with respect to fixed atlases
for $M$ and $N.$ If the atlases are both finite, this leads to a notion of $%
C^{p}$--distance, which depends on the atlases, but will serve our purposes.

For bundle maps and tensors we will need a $C^{p}$--norm, which we now
define. Recall that a Euclidean metric on a vector bundle $E$ restricts to
an inner product on each fiber of $E$ and these inner products vary
smoothly. Given vector bundles $E_{1}$ and $E_{2}$ with Euclidean metrics
and a bundle map 
\begin{equation*}
\varphi :E_{1}\longrightarrow E_{2},
\end{equation*}%
we define the $C^{p}$--norm of $\varphi ,$ $\left\vert \varphi \right\vert
_{C^{p}},$ as follows. Let $E_{1}^{1}$ be the unit sphere bundle of $E_{1}.$
Define $\left\vert \varphi \right\vert _{C^{p}}$ to be the $C^{p}$--distance
from $\varphi |_{E_{1}^{1}}$ to the zero bundle map. The $C^{p}$--norm of a
tensor is its $C^{p}$--distance to the zero section. We note that the $C^{p}$%
--norm of a bundle map or tensor depends on the given Euclidean metrics.
With the exception of $TM$, all of the vector bundles we consider will come
with a clear choice of metric. For bundle maps $\varphi $ that go to or from 
$TM$ and for tensors $\omega $ on $M,$ we adopt the convention that $%
\left\vert \varphi \right\vert _{C^{p}}$ and $\left\vert \omega \right\vert
_{C^{p}}$ are defined in terms of our initial $G$--invariant metric $g_{M}.$

\section{Regular Structure Theorem}

The vertical space for $q$ at $(g,x)\in G\times M$ is%
\begin{equation*}
\mathcal{V}=\{(-K_{G}\left( k\right) ,K_{M}\left( k\right) )\ |\ k\in 
\mathfrak{g}\}.
\end{equation*}

We recall from \cite{Cheeg}, \cite{PetWilh1}, \cite{SearleWilh} that there
is a linear reparametrization of the tangent space, called the \emph{Cheeger
reparametrization}. It is denoted by 
\begin{equation*}
Ch_{l}:TM\rightarrow TM
\end{equation*}%
and defined as 
\begin{equation*}
Ch_{l}\left( v\right) =Dq\left( \hat{v}_{l}\right) ,
\end{equation*}%
where $\hat{v}_{l}\in TG\times TM$ is the horizontal vector for 
\begin{equation*}
q:\left( G\times M,l^{2}g_{\mathrm{bi}}+g_{M}\right) \longrightarrow \left(
M,g_{l}\right)
\end{equation*}%
that projects to $v$ under the projection $\pi _{2}:$ $G\times
M\longrightarrow M.$

Although, $\hat{v}_{l}$ is completely determined by $v,$ $g_{\mathrm{bi}},$ $%
g_{M},$ and the $G$--action, the explicit formula is rather unpleasant, \cite%
{Muet}, \cite{Zil}. Fortunately, we will not need it, as we will use
abstract, asymptotic arguments.

Every $G$--orbit in $G\times M$ has a unique point of the form $(e,m).$ To
fix notation, we assume throughout that we are at such a point. When $l=1$
and $v\in T_{x}M,$ we denote the first factor of $\hat{v}_{1}$ by $\kappa
_{x}\left( v\right) .$ Then 
\begin{equation}
\hat{v}_{1}=\left( \kappa _{x}\left( v\right) ,v\right) .
\end{equation}%
For any $l,$ we then have 
\begin{equation*}
\hat{v}_{l}=\left( \frac{\kappa _{x}\left( v\right) }{l^{2}},v\right) .
\end{equation*}%
For simplicity, we will write $\hat{v}$ for $\hat{v}_{l}.$

\begin{proposition}
\label{K and Kappa}For $x\in M^{\text{reg}}$ we have the following.

\noindent 1. $K_{M,x}|_{\mathfrak{m}_{x}}:\mathfrak{m}_{x}\longrightarrow
T_{x}G\left( x\right) $ is an isomorphism that varies smoothly with $x.$

\noindent 2. The map $\kappa _{x}:T_{x}M\longrightarrow \mathfrak{g}_{x},$
given by $v\mapsto \kappa _{x}\left( v\right) ,$ takes values in $\mathfrak{m%
}_{x}$ and restricts to a linear isomorphism, $T_{x}G\left( x\right)
\longrightarrow \mathfrak{m}_{x},$ that varies smoothly with $x\in M^{\text{%
reg}}.$
\end{proposition}

\begin{proof}
Part $1$ follows from the definition of $K_{M,x}.$

Suppose $\left( u,v\right) \in T_{\left( e,x\right) }\left( G\times M\right) 
$ with $u\notin \mathfrak{m}_{x}.$ Then there is a $k\in \mathfrak{g}_{x}$
with $g_{\mathrm{bi}}\left( k,u\right) \neq 0.$ It follows that 
\begin{eqnarray*}
\left( l^{2}g_{\mathrm{bi}}+g_{M}\right) \left( \left( u,v\right) ,\left(
-K_{G,e}\left( k\right) ,K_{M,x}\left( k\right) \right) \right) &=&\left(
l^{2}g_{\mathrm{bi}}+g_{M}\right) \left( \left( u,v\right) ,\left(
-k,0\right) \right) \\
&\neq &0.
\end{eqnarray*}%
So $\left( u,v\right) $ is not horizontal. It follows that $\kappa _{x}$
takes values in $\mathfrak{m}_{x}.$ $\kappa _{x}$ is linear, since $%
Ch_{l}:T_{x}M\longrightarrow T_{x}M$ is linear and $\kappa _{x}$ is
projection to $G$ composed with $Ch_{l}|_{T_{x}M}$.

For $v\in TG\left( x\right) ,$ if $\left( 0,v\right) \in T\left( G\times
M\right) $ is horizontal, then $v=0,$ and it follows that $\kappa _{x}$ is
injective. Since $\mathrm{\dim }\left( \mathfrak{m}_{x}\right) =\dim \left(
G\left( x\right) \right) ,$ $\kappa _{x}:T_{x}G\left( x\right)
\longrightarrow \mathfrak{m}_{x}$ is, in fact, an isomorphism, proving Part
2.
\end{proof}

Before proceeding we define the following vector bundle over $M^{\text{reg}%
}. $ 
\begin{equation*}
E_{\mathrm{orb}}\equiv \left\{ \left. \left( x,v\right) \in M^{\text{reg}%
}\times \mathfrak{g}\text{ }\right\vert \text{ }v\in \mathfrak{m}%
_{x}\right\} .
\end{equation*}%
$K$ and $\kappa $ are then bundle maps%
\begin{eqnarray*}
&&E_{\mathrm{orb}}\overset{K_{M}}{\longrightarrow }\left. T\left( \mathrm{%
orbits}\right) \right\vert _{M^{\text{reg}}}, \\
&&\left. T\left( \mathrm{orbits}\right) \right\vert _{M^{\text{reg}}}\overset%
{\kappa }{\longrightarrow }E_{\mathrm{orb}}.
\end{eqnarray*}

\begin{proposition}
\label{C-p bounds prop}Given any compact subset $\mathcal{K}\subset M^{\text{%
reg}}$ and any $p\geq 0$ there is a constant $C>0$ so that 
\begin{equation*}
\max \left\{ \left\vert K\right\vert _{C^{p}},\left\vert \kappa \right\vert
_{C^{p}},\left\vert K^{-1}\right\vert _{C^{p}},\left\vert \kappa
^{-1}\right\vert _{C^{p}}\right\} \leq C.
\end{equation*}
\end{proposition}

\begin{proof}
This follows from compactness of the corresponding unit sphere bundles and
the fact that $K$, $\kappa ,$ $K^{-1},$ and $\kappa ^{-1}$ are $C^{\infty }.$
\end{proof}

The next result shows that along the orbits $\tilde{g}_{l}$ is approximately 
$\left( K_{M,x}|_{\mathfrak{m}_{x}}^{-1}\right) ^{\ast }\left( g_{\mathrm{bi}%
}\right) ,$ and the error in this approximation has the form $l^{2}\mathcal{%
\tilde{E}}$ for some bounded, symmetric $\left( 0,2\right) $--tensor $%
\mathcal{\tilde{E}}.$

\begin{lemma}
\label{orbital change}Given any compact subset $\mathcal{K}\subset M^{\text{%
reg}},$ there is an $l_{0}>0$ so that for all $l\in \left( 0,l_{0}\right) $
there is a symmetric $\left( 0,2\right) $--tensor $\mathcal{\tilde{E}}$ and
a constant $C>0$ with the following property: 
\begin{eqnarray}
\tilde{g}_{l}|_{_{\left. T\left( \mathrm{orbits}\right) \right\vert _{%
\mathcal{K}}}} &=&\left( K_{M}|^{-1}\right) ^{\ast }\left( g_{\mathrm{bi}%
}\right) +l^{2}\mathcal{\tilde{E}}\text{ and \label{tildeg on orbit eqn}} \\
\left\vert \mathcal{\tilde{E}}\right\vert _{C^{p}} &\leq &C.  \notag
\end{eqnarray}
\end{lemma}

\begin{proof}
For $x\in \mathcal{K}\subset M^{\text{reg }}$ and $v,w\in \left. T\left( 
\mathrm{orbits}\right) \right\vert _{\mathcal{K}},$ using Equation \ref%
{difff of q} we find%
\begin{eqnarray}
l^{2}Ch_{l}\left( v\right) &=&Dq\left( l^{2}\left( \frac{\kappa \left(
v\right) }{l^{2}},v\right) \right)  \notag \\
&=&K_{M}\left( \kappa \left( v\right) \right) +l^{2}v.  \label{Ch_l}
\end{eqnarray}

The definition of $g_{l}$ and $Ch_{l}$ gives 
\begin{eqnarray}
\frac{1}{l^{2}}g_{l}\left( l^{2}Ch_{l}\left( v\right) ,l^{2}Ch_{l}\left(
w\right) \right) &=&l^{2}\left( l^{2}g_{\mathrm{bi}}+g_{M}\right) \left(
\left( \frac{\kappa \left( v\right) }{l^{2}},v\right) ,\left( \frac{\kappa
\left( w\right) }{l^{2}},w\right) \right)  \notag \\
&=&g_{\mathrm{bi}}\left( \kappa \left( v\right) ,\kappa \left( w\right)
\right) +l^{2}g_{M}\left( v,w\right) .  \label{metric}
\end{eqnarray}%
So 
\begin{equation}
\frac{1}{l^{2}}\left( l^{2}Ch_{l}\right) ^{\ast }\left( g_{l}|_{T\left( 
\mathrm{orbits}\right) }\right) =\left( \kappa \right) ^{\ast }\left( g_{%
\mathrm{bi}}\right) +l^{2}\left( g_{M}|_{T\left( \mathrm{orbits}\right)
}\right) .  \label{tensor eqn}
\end{equation}%
From Equation \ref{Ch_l} we have 
\begin{equation*}
l^{2}Ch_{l}=K_{M}\circ \kappa +l^{2}\mathrm{id.}
\end{equation*}%
Combining this with Proposition \ref{C-p bounds prop} we see for small
enough $l$, there is a bundle map 
\begin{equation*}
E:\left. T\left( \mathrm{orbits}\right) \right\vert _{M^{\text{reg}%
}}\longrightarrow \left. T\left( \mathrm{orbits}\right) \right\vert _{M^{%
\text{reg}}}
\end{equation*}%
so that 
\begin{equation}
\left( l^{2}Ch_{l}\right) ^{-1}=\kappa ^{-1}\circ K_{M}^{-1}+O\left(
l^{2}\right) E,  \label{Cheeg inv}
\end{equation}%
and 
\begin{equation}
\left\vert E\right\vert _{C^{p}}\leq 1.  \label{bound on E eqn}
\end{equation}

Combining Equations \ref{tensor eqn} and \ref{Cheeg inv} gives%
\begin{eqnarray*}
\frac{1}{l^{2}}g_{l}|_{T\left( \mathrm{orbits}\right) } &=&\left( \left(
l^{2}Ch_{l}\right) ^{-1}\right) ^{\ast }\left( \kappa \right) ^{\ast }\left(
g_{\mathrm{bi}}\right) +l^{2}\left( \left( l^{2}Ch_{l}\right) ^{-1}\right)
^{\ast }\left( g_{M}|_{T\left( \mathrm{orbits}\right) }\right) \\
&=&\left( K_{M}^{-1}\right) ^{\ast }\left( g_{\mathrm{bi}}\right) +O\left(
l^{2}\right) \left( E\right) ^{\ast }\left( \kappa \right) ^{\ast }\left( g_{%
\mathrm{bi}}\right) + \\
&&+\,l^{2}\left( \kappa ^{-1}\circ K_{M}^{-1}\right) ^{\ast }\left(
g_{M}|_{T\left( \mathrm{orbits}\right) }\right) +\,O\left( l^{4}\right)
\left( E_{1}\right) ^{\ast }\left( g_{M}|_{T\left( \mathrm{orbits}\right)
}\right) \\
&=&\left( K_{M}^{-1}\right) ^{\ast }\left( g_{\mathrm{bi}}\right) +l^{2}%
\mathcal{\tilde{E}},
\end{eqnarray*}%
where 
\begin{equation*}
l^{2}\widetilde{\mathcal{E}}=O\left( l^{2}\right) \left( E\right) ^{\ast
}\left( \kappa \right) ^{\ast }\left( g_{\mathrm{bi}}\right) +\,l^{2}\left(
\kappa ^{-1}\circ K_{M}^{-1}\right) ^{\ast }\left( g_{M}|_{T\left( \mathrm{%
orbits}\right) }\right) +\,O\left( l^{4}\right) \left( E\right) ^{\ast
}\left( g_{M}|_{T\left( \mathrm{orbits}\right) }\right) .
\end{equation*}%
Combining this with Proposition \ref{C-p bounds prop} and Inequality \ref%
{bound on E eqn} it follows that%
\begin{equation*}
\left\vert \widetilde{\mathcal{E}}\right\vert _{C^{p}}\leq C
\end{equation*}%
for some $C>0.$
\end{proof}

\begin{proposition}
\label{diff}Given any compact subset $\mathcal{K}\subset M^{\text{reg}},$
there is an $l_{0}>0$ so that for all $l\in \left( 0,l_{0}\right) $ there is
a $\left( 0,2\right) $--symmetric tensor $\mathcal{E}$ and a constant $C>0$
with the following properties. For all $x\in \mathcal{K}$%
\begin{equation}
\left( \Phi _{x}\right) ^{\ast }\left( \tilde{g}_{l}\right) =g_{\mathrm{nh,x}%
}+l^{2}\mathcal{E}\text{ and}\text{ }  \label{Gap}
\end{equation}%
\begin{equation*}
\left\vert \mathcal{E}\right\vert _{C^{p}}\leq C.
\end{equation*}
\end{proposition}

\begin{proof}
Since $\Phi _{x}^{\ast }\left( \tilde{g}_{l}\right) $ and $g_{\mathrm{nh,x}}$
are $G$--invariant, it suffices to verify Equation \ref{Gap} at $eG_{x}.$
Using Equation \ref{DPhi eqn} and the linearity of $K_{M,x}$ and $K_{G/G_{x},%
\text{ }eG_{x}}^{-1},$ we see that applying $\left( \Phi _{x}\right) ^{\ast
} $ to Equation \ref{tildeg on orbit eqn} gives 
\begin{eqnarray*}
\left( \Phi _{x}\right) ^{\ast }\left( \tilde{g}_{l}|_{T_{x}G\left( x\right)
}\right) &=&\left( \Phi _{x}\right) ^{\ast }\left( K_{M,x}|_{\mathfrak{m}%
_{x}}^{-1}\right) ^{\ast }\left( g_{\mathrm{bi}}\right) +l^{2}\left( \Phi
_{x}\right) ^{\ast }\left( \mathcal{\tilde{E}}\right) \\
&=&\left( K_{M,x}\circ K_{G/G_{x},\text{ }eG_{x}}^{-1}\right) ^{\ast }\left(
K_{M,x}|_{\mathfrak{m}_{x}}^{-1}\right) ^{\ast }\left( g_{\mathrm{bi}%
}\right) +l^{2}\left( K_{M,x}\circ K_{G/G_{x},\text{ }eG_{x}}^{-1}\right)
^{\ast }\left( \mathcal{\tilde{E}}\right) \\
&=&\left( K_{M,x}|_{\mathfrak{m}_{x}}^{-1}\circ K_{M,x}\circ K_{G/G_{x},%
\text{ }eG_{x}}^{-1}\right) ^{\ast }\left( g_{\mathrm{bi}}\right)
+l^{2}\left( K_{M,x}\circ K_{G/G_{x},\text{ }eG_{x}}^{-1}\right) ^{\ast
}\left( \mathcal{\tilde{E}}\right) \\
&=&\left( \left. K_{G/G_{x},\text{ }eG_{x}}\right\vert _{\mathfrak{m}%
_{x}}^{-1}\right) ^{\ast }\left( g_{\mathrm{bi}}\right) +l^{2}\left(
K_{M,x}\circ K_{G/G_{x},\text{ }eG_{x}}^{-1}\right) ^{\ast }\left( \mathcal{%
\tilde{E}}\right) \\
&=&g_{\mathrm{nh,x}}+l^{2}\left( K_{M,x}\circ K_{G/G_{x},\text{ }%
eG_{x}}^{-1}\right) ^{\ast }\left( \mathcal{\tilde{E}}\right)
\end{eqnarray*}

The result then follows by setting 
\begin{equation*}
\mathcal{E}_{x}=\left( K_{M,x}\circ K_{G/G_{x},\text{ }eG_{x}}^{-1}\right)
^{\ast }\left( \mathcal{\tilde{E}}_{x}\right)
\end{equation*}%
and by appealing to Proposition \ref{C-p bounds prop} and the fact that $%
\left\vert \mathcal{\tilde{E}}\right\vert _{C^{p}}\leq C.$
\end{proof}

We are now in a position to begin the proofs of Theorems \ref{vert rescale}
and \ref{vert re-scl II thm}. First observe that the distribution orthogonal
to the orbits 
\begin{equation*}
x\mapsto TG\left( x\right) ^{\perp }
\end{equation*}%
is the same for $g_{l},$ $\tilde{g}_{l},$ and $g_{M}.$ Also notice that for $%
Z\in TG\left( x\right) ^{\perp },$ 
\begin{equation}
g_{l}\left( Z,\cdot \right) =\tilde{g}_{l}\left( Z,\cdot \right)
=g_{M}\left( Z,\cdot \right) .  \label{horiz static}
\end{equation}

For $x\in \mathcal{K}\subset M^{\text{reg}}$ we set 
\begin{equation}
\tilde{g}|_{x}\equiv g_{M}|_{TG\left( x\right) ^{\perp }}+\left( \Phi
_{x}^{-1}\right) ^{\ast }\left( g_{\mathrm{nh,x}}\right) .
\label{dfn of
tilde-g}
\end{equation}%
Our next result shows that $\tilde{g}$ is $G$--invariant.

\begin{proposition}
For $y\in G\left( x\right) ,$ $\left( \Phi _{x}^{-1}\right) ^{\ast }\left(
g_{\mathrm{nh,x}}\right) =\left( \Phi _{y}^{-1}\right) ^{\ast }\left( g_{%
\mathrm{nh,y}}\right) $
\end{proposition}

\begin{proof}
Let $g_{yx}\in G$ satisfy $g_{yx}x=y.$ Then $g_{yx}G_{x}g_{yx}^{-1}=G_{y}$
and we have a commutative diagram%
\begin{equation*}
\begin{xy}
(0,40)*+{G}="v1";
(30,40)*+{G}="v2";%
(0,20)*+{G/G_x}="v3";
(30,20)*+{G/G_y}="v4";
(0,0)*+{G(x)}="v5";%
(30,0)*+{G(y)}="v6";
{\ar "v1"; "v2"}?*!/_3mm/{\mathrm{C}_{g_{yx}}}; 
{\ar "v1"; "v3"}?*!/^4mm/{\pi _{G_{x}}};%
{\ar "v2"; "v4"}?*!/_4mm/{\pi _{G_{y}}};
{\ar "v3"; "v4"}?*!/_3mm/{\mathrm{\bar{C}}_{g_{yx}}}; 
{\ar "v3"; "v5"}?*!/^4mm/{\Phi_x};%
{\ar "v4"; "v6"}?*!/_4mm/{\Phi_y};
{\ar "v5"; "v6"}?*!/_3mm/{L_{g_{yx}}};
\end{xy}
\end{equation*}

where 
\begin{eqnarray*}
\mathrm{C}_{g_{yx}}\left( a\right) &=&g_{yx}ag_{yx}^{-1}, \\
\mathrm{\bar{C}}_{g_{yx}}\left( aG_{x}\right) &=&g_{yx}ag_{yx}^{-1}G_{y}, \\
L_{g_{yx}}\left( p\right) &=&g_{yx}p,
\end{eqnarray*}
and $\pi _{G_{x}}$ and $\pi _{G_{y}}$ are the quotient maps.

It follows that 
\begin{eqnarray*}
\left( \Phi _{x}^{-1}\right) ^{\ast }\left( g_{\mathrm{nh,x}}\right)
&=&\left( \left( \mathrm{\bar{C}}_{g_{yx}}\right) ^{-1}\circ \Phi
_{y}^{-1}\circ L_{g_{yx}}\right) ^{\ast }\left( g_{\mathrm{nh,x}}\right) \\
&=&\left( L_{g_{yx}}\right) ^{\ast }\circ \left( \Phi _{y}^{-1}\right)
^{\ast }\circ \left( \mathrm{\bar{C}}_{g_{yx}}^{-1}\right) ^{\ast }\left( g_{%
\mathrm{nh,x}}\right) \\
&=&\left( L_{g_{yx}}\right) ^{\ast }\circ \left( \Phi _{y}^{-1}\right)
^{\ast }\left( g_{\mathrm{nh,y}}\right) \\
&=&\left( \Phi _{y}^{-1}\right) ^{\ast }\left( g_{\mathrm{nh,y}}\right) ,
\end{eqnarray*}%
since $L_{g_{yx}}$ is an isometry of $\left( G\left( y\right) ,\left( \Phi
_{y}^{-1}\right) ^{\ast }\left( g_{\mathrm{nh,y}}\right) \right) .$
\end{proof}

Applying $\left( \Phi _{x}^{-1}\right) ^{\ast }$ to both sides of Equation %
\ref{Gap}, we obtain 
\begin{equation}
\text{ }\tilde{g}_{l}|_{TG\left( x\right) }=\left( \Phi _{x}^{-1}\right)
^{\ast }\left( g_{\mathrm{nh,x}}\right) +l^{2}\left( \Phi _{x}^{-1}\right)
^{\ast }\left( \mathcal{E}\right) .  \label{vert conv}
\end{equation}%
\bigskip

Combining Equations \ref{DPhi eqn}, \ref{horiz static} and \ref{vert conv}
with the inequality, $\left\vert \mathcal{E}\right\vert _{C^{p}}\leq C,$ we
see that 
\begin{equation}
\left\vert \tilde{g}_{l}-\tilde{g}\right\vert _{C^{p}}\leq Cl^{2}.
\label{C-p
close inequal}
\end{equation}

\begin{remark}
Our proof does not preclude the possibility that the bounds on the higher
order derivatives of $\mathcal{E}$ depend on the order $p,$ and so does not
give convergence in the $C^{\infty }$--topology.
\end{remark}
The next result shows that the fibers of $\pi ^{\text{reg}}:\left( \mathcal{U%
},\tilde{g}\right) \longrightarrow \mathcal{U}/G$ are totally geodesic and,
combined with Inequality \ref{C-p close inequal}, completes the proofs of
Theorems \ref{vert rescale} and \ref{vert re-scl II thm}.
\begin{proposition}
Let $T^{g_{M}}$ and $T^{\tilde{g}_{l}}$ be the $T$--tensors of the
Riemannian submersions 
\begin{eqnarray*}
\pi ^{\text{reg}} &:&\left( M^{\text{reg}},g_{M}\right) \longrightarrow M^{%
\text{reg}}/G,\text{ and} \\
\pi ^{\text{reg}} &:&\left( M^{\text{reg}},\tilde{g}_{l}\right)
\longrightarrow M^{\text{reg}}/G,
\end{eqnarray*}%
as defined in \cite{O'Neill}.
Given any compact subset $\mathcal{K}\subset M^{\text{reg }}$there is a
constant $C>0$ so that on $\mathcal{K}$ 
\begin{equation}
\left\vert T^{\tilde{g}_{l}}\right\vert \leq Cl^{2}\left\vert
T^{g_{M}}\right\vert .  \label{T}
\end{equation}
\end{proposition}

\begin{proof}
Let $T^{g_{l}}$ be the $T$--tensor of the Riemannian submersion 
\begin{equation*}
\pi ^{\text{reg}}:\left( M^{\text{reg}},g_{l}\right) \longrightarrow M^{%
\text{reg}}/G.
\end{equation*}%
The duality between the shape operator and the second fundamental form of
the fibers implies that the norm of the $T$--tensor is determined by its
values on just the vertical vectors.

We begin by proving Inequality \ref{T} with $T^{\tilde{g}_{l}}$ replaced by $%
T^{g_{l}}$ and then we will show that $\left\vert T^{\tilde{g}%
_{l}}\right\vert =\left\vert T^{g_{l}}\right\vert .$

For $V,W\in TG\left( x\right) $ and $Z\in TG\left( x\right) ^{\perp },$ we
lift $Ch_{l}\left( V\right) ,$ $Ch_{l}\left( W\right) ,$ and $Ch_{l}\left(
Z\right) $ to $G\times M$ and get 
\begin{eqnarray*}
g_{l}\left( T_{Ch_{l}\left( V\right) }Ch_{l}\left( W\right) ,Ch_{l}\left(
Z\right) \right) &=&\left( l^{2}g_{\mathrm{bi}}+g_{M}\right) \left( \nabla
_{\left( \frac{\kappa \left( V\right) }{l^{2}},V\right) }^{l^{2}g_{\mathrm{bi%
}}+g_{M}}\left( \frac{\kappa \left( W\right) }{l^{2}},W\right) ,\left(
0,Z\right) \right) \\
&=&g_{M}\left( \nabla _{V}^{g_{M}}W,Z\right) \\
&=&g_{M}\left( T_{V}^{g_{M}}W,Z\right)
\end{eqnarray*}%
On the other hand if $\left\vert V\right\vert _{g_{M}}=\left\vert
W\right\vert _{g_{M}}=1,$ then 
\begin{equation*}
\left\vert Ch_{l}\left( V\right) \right\vert ^{2}=\frac{\left\vert \kappa
\left( V\right) \right\vert _{g_{\mathrm{bi}}}^{2}}{l^{2}}+1\text{ and }%
\left\vert Ch_{l}\left( W\right) \right\vert ^{2}=\frac{\left\vert \kappa
\left( W\right) \right\vert _{g_{\mathrm{bi}}}^{2}}{l^{2}}+1.
\end{equation*}%
Combining the previous two displays with Proposition \ref{C-p bounds prop}
we see that given any compact subset $\mathcal{K}\subset M^{\text{reg }}$%
there is a constant $C>0$ so that 
\begin{equation*}
\left\vert T^{g_{l}}\right\vert \leq Cl^{2}\left\vert T^{g_{M}}\right\vert .
\end{equation*}%
To see $\left\vert T^{\tilde{g}_{l}}\right\vert =\left\vert
T^{g_{l}}\right\vert $ we use the Koszul formula and find that%
\begin{eqnarray*}
2\tilde{g}_{l}\left( T_{lV}^{\tilde{g}_{l}}lW,Z\right) &=&2l^{2}\tilde{g}%
_{l}\left( \tilde{\nabla}_{V}W,Z\right) \\
&=&l^{2}\left( -D_{Z}\tilde{g}_{l}\left( V,W\right) +\tilde{g}_{l}\left( %
\left[ Z,V\right] ,W\right) +\tilde{g}_{l}\left( \left[ Z,W\right] ,V\right)
\right) \\
&=&-D_{Z}g_{l}\left( V,W\right) +g_{l}\left( \left[ Z,V\right] ,W\right)
+g_{l}\left( \left[ Z,W\right] ,V\right) \\
&=&2g_{l}\left( T_{V}^{g_{l}}W,Z\right) .
\end{eqnarray*}%
So $\left\vert T^{\tilde{g}_{l}}\right\vert =\left\vert T^{g_{l}}\right\vert
,$ and the result follows.
\end{proof}

http://www.math.upenn.edu/\symbol{126}wziller/papers/SummaryMueter.pdf

\end{document}